\documentclass{amsart}
\usepackage{amssymb}
\usepackage{amsfonts}

\newtheorem{theorem}{Theorem}
\theoremstyle{plain}

\newtheorem{definition}{Definition}

\newtheorem{lemma}{Lemma}

\newtheorem{remark}{Remark}

\numberwithin{equation}{section}
\input{tcilatex}

\begin{document}
\title[Inequalities for product of convex functions]{ON THE HADAMARD TYPE
INEQUALITIES INVOLVING PRODUCT OF TWO CONVEX FUNCTIONS ON THE CO-ORDINATES}
\author{$^{\bigstar }$M. Emin \"{O}ZDEM\.{I}R}
\address{$^{\bigstar }$Atat\"{u}rk University, K. K. Education Faculty,
Department of Mathematics, 25240, Kampus, Erzurum, Turkey}
\email{emos@atauni.edu.tr}
\author{$^{\blacklozenge ,\clubsuit }$Ahmet Ocak AKDEM\.{I}R}
\address{$^{\blacklozenge }$A\u{g}r\i\ \.{I}brahim \c{C}e\c{c}en University,%
\\
Faculty of Science and Arts, Department of Mathematics, 04100, A\u{g}r\i ,
Turkey}
\email{ahmetakdemir@agri.edu.tr}
\date{September 2010}
\keywords{Convex, $s-$convex, Co-ordinates, Product of functions, Hadamard's
inequality, Beta function, Gamma function.\\
$^{\clubsuit }$Corresponding author}

\begin{abstract}
In this paper some Hadamard-type inequalities for product of convex
funcitons of $2-$variables on the co-ordinates are given.
\end{abstract}

\maketitle

\section{INTRODUCTION}

The inequality%
\begin{equation}
f\left( \frac{a+b}{2}\right) \leq \frac{1}{b-a}\dint\limits_{a}^{b}f(x)dx%
\leq \frac{f(a)+f(b)}{2}  \label{1.1}
\end{equation}%
where $f:I\subset 
%TCIMACRO{\U{211d} }%
%BeginExpansion
\mathbb{R}
%EndExpansion
\rightarrow 
%TCIMACRO{\U{211d} }%
%BeginExpansion
\mathbb{R}
%EndExpansion
$ is a convex function defined on the interval $I$ of $%
%TCIMACRO{\U{211d} }%
%BeginExpansion
\mathbb{R}
%EndExpansion
,$ the set of real numbers, and $a,b\in I$ with $a<b,$ is well known in the
literature as Hadamard's inequality.

For some recent results related to this classic inequality, see \cite{PP}, 
\cite{MEO}, \cite{MU}, \cite{SS}, and \cite{DA}, where further references
are given.

In \cite{HM}, Hudzik and Maligranda considered, among others, the class of
functions which are $s-$convex in the second sense. This class is defined as
following:

\begin{definition}
A function $f:[0,\infty )\rightarrow $ $%
%TCIMACRO{\U{211d} }%
%BeginExpansion
\mathbb{R}
%EndExpansion
$ is said to be $s-$convex in the second sense if%
\begin{equation*}
f(\lambda x+(1-\lambda )y)\leq \lambda ^{s}f(x)+(1-\lambda )^{s}f(y)
\end{equation*}%
holds for all $x,y\in \lbrack 0,\infty ),$ $\lambda \in \lbrack 0,1]$ and
for some fixed $s$ $\in (0,1].$
\end{definition}

The class of $s-$convex functions in the second sense is usually denoted
with $K_{s}^{2}.$ It is clear that if we choose $s=1$ we have ordinary
convexity of functions defined on $[0,\infty ).$

In \cite{UK}, K\i rmac\i\ \textit{et al.}, proved the following inequalities
related to product of convex functions. These are given in the next theorems.

\begin{theorem}
Let $f,g:[a,b]\rightarrow 
%TCIMACRO{\U{211d} }%
%BeginExpansion
\mathbb{R}
%EndExpansion
,a,b\in \lbrack 0,\infty ),$ $a<b,$ be functions such that $g$ and $fg$ are
in $L^{1}([a,b]).$ If $f$ is convex and nonnegative on $[a,b],$ and if $g$
is $s-$convex on $[a,b]$ for some fixed $s$ $\in (0,1),$ then%
\begin{equation}
\frac{1}{b-a}\dint\limits_{a}^{b}f(x)g(x)dx\leq \frac{1}{s+2}M(a,b)+\frac{1}{%
(s+1)(s+2)}N(a,b)  \label{1.2}
\end{equation}%
where%
\begin{equation*}
M(a,b)=f(a)g(a)+f(b)g(b)\text{ and }N(a,b)=f(a)g(b)+f(b)g(a).
\end{equation*}
\end{theorem}

\begin{theorem}
Let $f,g:[a,b]\rightarrow 
%TCIMACRO{\U{211d} }%
%BeginExpansion
\mathbb{R}
%EndExpansion
,$ $a,b\in \lbrack 0,\infty ),$ $a<b,$ be functions such that $g$ and $fg$
are in $L^{1}([a,b]).$ If $f$ is $s_{1}-$convex and $g$ is $s_{2}-$convex on 
$[a,b]$ for some fixed $s_{1},s_{2}\in (0,1),$ then%
\begin{eqnarray}
\frac{1}{b-a}\dint\limits_{a}^{b}f(x)g(x)dx &\leq &\frac{1}{s_{1}+s_{2}+1}%
M(a,b)+B(s_{1}+1,s_{2}+1)N(a,b)  \notag \\
&=&\frac{1}{s_{1}+s_{2}+1}\left[ M(a,b)+s_{1}s_{2}\frac{\Gamma (s_{1})\Gamma
(s_{2})}{\Gamma (s_{1}+s_{2}+1)}N(a,b)\right]  \label{1.3}
\end{eqnarray}
\end{theorem}

\begin{theorem}
Let $f,g:[a,b]\rightarrow 
%TCIMACRO{\U{211d} }%
%BeginExpansion
\mathbb{R}
%EndExpansion
,$ $a,b\in \lbrack 0,\infty ),$ $a<b,$ be functions such that $g$ and $fg$
are in $L^{1}([a,b]).$ If $f$ is convex and nonnegative on $[a,b],$ and if $%
g $ is $s-$convex on $[a,b]$ for some fixed $s$ $\in (0,1),$ then%
\begin{eqnarray}
&&2^{s}f(\frac{a+b}{2})g(\frac{a+b}{2})-\frac{1}{b-a}\dint%
\limits_{a}^{b}f(x)g(x)dx  \label{1.4} \\
&\leq &\frac{1}{(s+1)(s+2)}M(a,b)+\frac{1}{s+2}N(a,b)  \notag
\end{eqnarray}
\end{theorem}

For similar results, see the papers \cite{HM}, \cite{DF}.

In \cite{SS}, Dragomir defined convex functions on the co-ordinates as
follows and proved Lemma 1 related to this definiton:

\begin{definition}
Let us consider the bidimensional interval $\Delta =[a,b]\times \lbrack c,d]$
in $%
%TCIMACRO{\U{211d} }%
%BeginExpansion
\mathbb{R}
%EndExpansion
^{2}$ with $a<b,$ $c<d.$ A function $f:\Delta \rightarrow 
%TCIMACRO{\U{211d} }%
%BeginExpansion
\mathbb{R}
%EndExpansion
$ will be called convex on the co-ordinates if the partial mappings $%
f_{y}:[a,b]\rightarrow 
%TCIMACRO{\U{211d} }%
%BeginExpansion
\mathbb{R}
%EndExpansion
,$ $f_{y}(u)=f(u,y)$ and $f_{x}:[c,d]\rightarrow 
%TCIMACRO{\U{211d} }%
%BeginExpansion
\mathbb{R}
%EndExpansion
,$ $f_{x}(v)=f(x,v)$ are convex where defined for all $y\in \lbrack c,d]$
and $x\in \lbrack a,b].$ Recall that the mapping $f:\Delta \rightarrow 
%TCIMACRO{\U{211d} }%
%BeginExpansion
\mathbb{R}
%EndExpansion
$ is convex on $\Delta $ if the following inequality holds, 
\begin{equation*}
f(\lambda x+(1-\lambda )z,\lambda y+(1-\lambda )w)\leq \lambda
f(x,y)+(1-\lambda )f(z,w)
\end{equation*}%
for all $(x,y),(z,w)\in \Delta $ and $\lambda \in \lbrack 0,1].$
\end{definition}

\begin{lemma}
Every \ convex mapping $f:\Delta \rightarrow 
%TCIMACRO{\U{211d} }%
%BeginExpansion
\mathbb{R}
%EndExpansion
$ is convex on the co-ordinates, but converse is not general true.
\end{lemma}

A formal definition for co-ordinated convex functions may be stated as
follow [see \cite{MAL}]:

\begin{definition}
A function $f:\Delta \rightarrow \mathbb{R}$ is said to be convex on the
co-ordinates on $\Delta $ if the following inequality:%
\begin{align*}
& f(tx+(1-t)y,su+(1-s)w) \\
& \leq tsf(x,u)+t(1-s)f(x,w)+s(1-t)f(y,u)+(1-t)(1-s)f(y,w)
\end{align*}%
holds for all $t,s\in \lbrack 0,1]$ and $(x,u),(x,w),(y,u),(y,w)\in \Delta $.
\end{definition}

In \cite{SS}, Dragomir established the following inequalities:

\begin{theorem}
Suppose that $f:\Delta =$ $[a,b]\times \lbrack c,d]\rightarrow 
%TCIMACRO{\U{211d} }%
%BeginExpansion
\mathbb{R}
%EndExpansion
$ is convex on the co-ordinates on $\Delta .$ Then one has the inequalities:%
\begin{eqnarray}
&&f(\frac{a+b}{2},\frac{c+d}{2})  \notag \\
&\leq &\frac{1}{\left( b-a\right) \left( d-c\right) }\dint\limits_{a}^{b}%
\dint\limits_{c}^{d}f(x,y)dxdy  \label{1.5} \\
&\leq &\frac{f(a,c)+f(a,d)+f(b,c)+f(b,d)}{4}  \notag
\end{eqnarray}
\end{theorem}

Similar results, refinements and generalizations can be found in \cite{LA}, 
\cite{DAR1}, \cite{DAR2}, \cite{DAR3}, \cite{MSET} and \cite{OZ2}.

In \cite{DAR3}, Alomari and Darus defined $s-$convexity on $\Delta $ as
follows:

\begin{definition}
Consider the bidimensional interval $\Delta :=$ $[a,b]\times \lbrack c,d]$
in $[0,\infty )^{2}$ with $a<b$ and $c<d.$ The mapping $f:\Delta \rightarrow 
%TCIMACRO{\U{211d} }%
%BeginExpansion
\mathbb{R}
%EndExpansion
$ is $s-$convex on $\Delta $ if%
\begin{equation*}
f(\lambda x+(1-\lambda )z,\lambda y+(1-\lambda )w)\leq \lambda
^{s}f(x,y)+(1-\lambda )^{s}f(z,w)
\end{equation*}%
holds for all $(x,y),$ $(z,w)\in \Delta $ with $\lambda \in \lbrack 0,1]$
and for some fixed $s$ $\in (0,1].$
\end{definition}

In \cite{DAR3}, Alomari and Darus proved the following lemma:

\begin{lemma}
Every $s-$convex mappings $f:\Delta :=$ $[a,b]\times \lbrack c,d]\subset
\lbrack 0,\infty )^{2}\rightarrow \lbrack 0,\infty )$ is $s-$convex on the
co-ordinates, but converse is not true in general .
\end{lemma}

In \cite{LAT}, Latif and Alomari established Hadamard-type inequalities for
product of two convex functions on the co-ordinates as follow:

\begin{theorem}
Let $\ f,g:\Delta :=$ $[a,b]\times \lbrack c,d]\subset 
%TCIMACRO{\U{211d} }%
%BeginExpansion
\mathbb{R}
%EndExpansion
^{2}\rightarrow \lbrack 0,\infty )$ be convex functions on the co-ordinates
on $\Delta $ with $a<b$ and $c<d.$ Then 
\begin{eqnarray}
&&\frac{1}{\left( b-a\right) \left( d-c\right) }\dint\limits_{a}^{b}\dint%
\limits_{c}^{d}f(x,y)g(x,y)dxdy  \label{1.6} \\
&\leq &\frac{1}{9}L(a,b,c,d)+\frac{1}{18}M(a,b,c,d)+\frac{1}{36}N(a,b,c,d) 
\notag
\end{eqnarray}%
where%
\begin{eqnarray*}
L(a,b,c,d) &=&f(a,c)g(a,c)+f(b,c)g(b,c)+f(a,d)g(a,d)+f(b,d)g(b,d) \\
M(a,b,c,d) &=&f(a,c)g(a,d)+f(a,d)g(a,c)+f(b,c)g(b,d)+f(b,d)g(b,c) \\
&&+f(b,c)g(a,c)+f(b,d)g(a,d)+f(a,c)g(b,c)+f(a,d)g(b,d) \\
N(a,b,c,d) &=&f(b,c)g(a,d)+f(b,d)g(a,c)+f(a,c)g(b,d)+f(a,d)g(b,c)
\end{eqnarray*}
\end{theorem}

\begin{theorem}
Let $\ f,g:\Delta :=$ $[a,b]\times \lbrack c,d]\subset 
%TCIMACRO{\U{211d} }%
%BeginExpansion
\mathbb{R}
%EndExpansion
^{2}\rightarrow \lbrack 0,\infty )$ be convex functions on the co-ordinates
on $\Delta $ with $a<b$ and $c<d.$ Then%
\begin{eqnarray}
&&4f(\frac{a+b}{2},\frac{c+d}{2})g(\frac{a+b}{2},\frac{c+d}{2})  \label{1.7}
\\
&\leq &\frac{1}{\left( b-a\right) \left( d-c\right) }\dint\limits_{a}^{b}%
\dint\limits_{c}^{d}f(x,y)g(x,y)dxdy  \notag \\
&&+\frac{5}{36}L(a,b,c,d)+\frac{7}{36}M(a,b,c,d)+\frac{2}{9}N(a,b,c,d) 
\notag
\end{eqnarray}%
where $L(a,b,c,d),$ $M(a,b,c,d)$, $N(a,b,c,d)$ as in (\ref{1.6}).
\end{theorem}

The main purpose of this paper is to establish new inequalities like (\ref%
{1.6}) and (\ref{1.7}), but now for convex functions and $s-$convex
functions of $2-$variables on the co-ordinates.

\section{MAIN RESULTS}

\begin{theorem}
Let $\ f:\Delta :=$ $[a,b]\times \lbrack c,d]\subset \lbrack 0,\infty
)^{2}\rightarrow \lbrack 0,\infty )$ be convex function on the co-ordinates
and $g:\Delta :=$ $[a,b]\times \lbrack c,d]\subset \lbrack 0,\infty
)^{2}\rightarrow \lbrack 0,\infty )$ be $s-$convex function on the
co-ordinates with $a<b,$ $c<d$ and $f_{x}(y)g_{x}(y),$ $\
f_{y}(x)g_{y}(x)\in L^{1}[\Delta ]$ for some fixed $s\in (0,1].$ Then one
has the inequality:%
\begin{eqnarray}
&&\frac{1}{\left( d-c\right) \left( b-a\right) }\dint\limits_{a}^{b}\dint%
\limits_{c}^{d}f(x,y)g(x,y)dxdy  \label{2.1} \\
&\leq &\frac{1}{\left( s+2\right) ^{2}}L(a,b,c,d)+\frac{1}{(s+1)\left(
s+2\right) ^{2}}M(a,b,c,d)  \notag \\
&&+\frac{1}{(s+1)^{2}(s+2)^{2}}N(a,b,c,d)  \notag
\end{eqnarray}%
where%
\begin{eqnarray*}
L(a,b,c,d) &=&\frac{1}{\left( s+2\right) ^{2}}\left( \left[
f(a,c)g(a,c)+f(b,c)g(b,c)\right] +\left[ f(a,d)g(a,d)+f(b,d)g(b,d)\right]
\right) \\
M(a,b,c,d) &=&\frac{1}{(s+1)\left( s+2\right) ^{2}}\left( \left[
f(a,c)g(b,c)+f(b,c)g(a,c)\right] +\left[ f(a,d)g(b,d)+f(b,d)g(a,d)\right]
\right) \\
&&+\frac{1}{(s+1)\left( s+2\right) ^{2}}\left( \left[
f(a,c)g(a,d)+f(b,c)g(b,d)\right] +\left[ f(a,d)g(a,c)+f(b,d)g(b,c)\right]
\right) \\
N(a,b,c,d) &=&\frac{1}{(s+1)^{2}(s+2)^{2}}\left( \left[
f(a,c)g(b,d)+f(b,c)g(a,d)\right] +\left[ f(a,d)g(b,c)+f(b,d)g(a,c)\right]
\right)
\end{eqnarray*}
\end{theorem}

\begin{proof}
Since $f$ is co-ordinated convex and $g$ is co-ordinated $s-$convex, from
Lemma 1 and Lemma 2, the partial mappings%
\begin{eqnarray*}
f_{y} &:&[a,b]\rightarrow \lbrack 0,\infty ),\text{ }f_{y}(x)=f(x,y),\text{ }%
y\in \lbrack c,d] \\
f_{x} &:&[c,d]\rightarrow \lbrack 0,\infty ),\text{ }f_{x}(y)=f(x,y),\text{ }%
x\in \lbrack a,b]
\end{eqnarray*}%
are convex on $[a,b]$ and $[c,d],$ respectively, where $x\in \lbrack a,b],$ $%
y\in \lbrack c,d]$. Similarly;%
\begin{eqnarray*}
g_{y} &:&[a,b]\rightarrow \lbrack 0,\infty ),\text{ }g_{y}(x)=g(x,y),\text{ }%
y\in \lbrack c,d] \\
g_{x} &:&[c,d]\rightarrow \lbrack 0,\infty ),\text{ }g_{x}(y)=g(x,y),\text{ }%
x\in \lbrack a,b]
\end{eqnarray*}%
are $s-$convex on $[a,b]$ and $[c,d],$ respectively, where $x\in \lbrack
a,b],$ $y\in \lbrack c,d].$

Using (\ref{1.2}), we can write%
\begin{eqnarray*}
\frac{1}{d-c}\dint\limits_{c}^{d}f_{x}(y)g_{x}(y)dy &\leq &\frac{1}{s+2}%
\left[ f_{x}(c)g_{x}(c)+f_{x}(d)g_{x}(d)\right] \\
&&+\frac{1}{(s+1)(s+2)}\left[ f_{x}(c)g_{x}(d)+f_{x}(d)g_{x}(c)\right]
\end{eqnarray*}%
That is%
\begin{eqnarray*}
\frac{1}{d-c}\dint\limits_{c}^{d}f(x,y)g(x,y)dy &\leq &\frac{1}{s+2}\left[
f(x,c)g(x,c)+f(x,d)g(x,d)\right] \\
&&+\frac{1}{(s+1)(s+2)}\left[ f(x,c)g(x,d)+f(x,d)g(x,c)\right]
\end{eqnarray*}%
Dividing both sides by$\ (b-a)$ and integrating over $[a,b],$ we get%
\begin{eqnarray}
&&\frac{1}{\left( d-c\right) \left( b-a\right) }\dint\limits_{a}^{b}\dint%
\limits_{c}^{d}f(x,y)g(x,y)dxdy  \label{2.2} \\
&\leq &\frac{1}{s+2}\left[ \frac{1}{b-a}\dint\limits_{a}^{b}f(x,c)g(x,c)dx+%
\frac{1}{b-a}\dint\limits_{a}^{b}f(x,d)g(x,d)dx\right]  \notag \\
&&+\frac{1}{(s+1)(s+2)}\left[ \frac{1}{b-a}\dint%
\limits_{a}^{b}f(x,c)g(x,d)dx+\frac{1}{b-a}\dint\limits_{a}^{b}f(x,d)g(x,c)dx%
\right]  \notag
\end{eqnarray}%
By applying (\ref{1.2}) to each term of right hand side of above inequality,
we have%
\begin{eqnarray*}
\frac{1}{b-a}\dint\limits_{a}^{b}f(x,c)g(x,c)dx &\leq &\frac{1}{s+2}\left[
f(a,c)g(a,c)+f(b,c)g(b,c)\right] \\
&&+\frac{1}{(s+1)(s+2)}\left[ f(a,c)g(b,c)+f(b,c)g(a,c)\right]
\end{eqnarray*}%
\begin{eqnarray*}
\frac{1}{b-a}\dint\limits_{a}^{b}f(x,d)g(x,d)dx &\leq &\frac{1}{s+2}\left[
f(a,d)g(a,d)+f(b,d)g(b,d)\right] \\
&&+\frac{1}{(s+1)(s+2)}\left[ f(a,d)g(b,d)+f(b,d)g(a,d)\right]
\end{eqnarray*}%
\begin{eqnarray*}
\frac{1}{b-a}\dint\limits_{a}^{b}f(x,c)g(x,d)dx &\leq &\frac{1}{s+2}\left[
f(a,c)g(a,d)+f(b,c)g(b,d)\right] \\
&&+\frac{1}{(s+1)(s+2)}\left[ f(a,c)g(b,d)+f(b,c)g(a,d)\right]
\end{eqnarray*}%
\begin{eqnarray*}
\frac{1}{b-a}\dint\limits_{a}^{b}f(x,d)g(x,c)dx &\leq &\frac{1}{s+2}\left[
f(a,d)g(a,c)+f(b,d)g(b,c)\right] \\
&&+\frac{1}{(s+1)(s+2)}\left[ f(a,d)g(b,c)+f(b,d)g(a,c)\right]
\end{eqnarray*}%
Using these inequalities in (\ref{2.2}), (\ref{2.1}) is proved, that is 
\begin{eqnarray*}
&&\frac{1}{\left( d-c\right) \left( b-a\right) }\dint\limits_{a}^{b}\dint%
\limits_{c}^{d}f(x,y)g(x,y)dxdy \\
&\leq &\frac{1}{\left( s+2\right) ^{2}}\left( \left[
f(a,c)g(a,c)+f(b,c)g(b,c)\right] +\left[ f(a,d)g(a,d)+f(b,d)g(b,d)\right]
\right) \\
&&+\frac{1}{(s+1)\left( s+2\right) ^{2}}\left( \left[
f(a,c)g(b,c)+f(b,c)g(a,c)\right] +\left[ f(a,d)g(b,d)+f(b,d)g(a,d)\right]
\right) \\
&&+\frac{1}{(s+1)\left( s+2\right) ^{2}}\left( \left[
f(a,c)g(a,d)+f(b,c)g(b,d)\right] +\left[ f(a,d)g(a,c)+f(b,d)g(b,c)\right]
\right) \\
&&+\frac{1}{(s+1)^{2}(s+2)^{2}}\left( \left[ f(a,c)g(b,d)+f(b,c)g(a,d)\right]
+\left[ f(a,d)g(b,c)+f(b,d)g(a,c)\right] \right)
\end{eqnarray*}%
We can find the same result using by $f_{y}(x)$ and $g_{y}(x).$
\end{proof}

\begin{remark}
In (\ref{2.1}), if we choose $s=1$, (\ref{1.6}) is obtained.
\end{remark}

\begin{remark}
In (\ref{2.1}), if we choose $s=1$ and $f(x)=1$ which is convex, we get the
second inequality in (\ref{1.5}) :%
\begin{equation*}
\frac{1}{\left( d-c\right) \left( b-a\right) }\dint\limits_{a}^{b}\dint%
\limits_{c}^{d}g(x,y)dxdy\leq \frac{g(a,c)+g(b,c)+g(a,d)+g(b,d)}{4}
\end{equation*}
\end{remark}

In the next theorem we will also make use of the Beta function of Euler
type, which is for $x,y>0$ defined as%
\begin{equation*}
B(x,y)=\dint\limits_{0}^{1}t^{x-1}(1-t)^{y-1}dt=\frac{\Gamma (x)\Gamma (y)}{%
\Gamma (x+y)}
\end{equation*}%
and the Gamma function is defined as%
\begin{equation*}
\Gamma (x)=\dint\limits_{0}^{\infty }t^{x-1}e^{-t}dt,\text{ for }x>0.
\end{equation*}

\begin{theorem}
Let $\ f:\Delta :=$ $[a,b]\times \lbrack c,d]\subset \lbrack 0,\infty
)^{2}\rightarrow \lbrack 0,\infty )$ be $s_{1}-$convex function on the
co-ordinates and $g:\Delta :=$ $[a,b]\times \lbrack c,d]\subset \lbrack
0,\infty )^{2}\rightarrow \lbrack 0,\infty )$ be $s_{2}-$convex functions on
the co-ordinates with $a<b,$ $c<d$ and $f_{x}(y)g_{x}(y),$ $%
f_{y}(x)g_{y}(x)\in L^{1}[\Delta ]$ for some fixed $s_{1},s_{2}\in (0,1].$
Then one has the inequality:%
\begin{eqnarray}
&&\frac{1}{\left( d-c\right) \left( b-a\right) }\dint\limits_{a}^{b}\dint%
\limits_{c}^{d}f(x,y)g(x,y)dxdy  \label{2.3} \\
&\leq &\frac{1}{\left( s_{1}+s_{2}+1\right) ^{2}}L(a,b,c,d)+\frac{%
B(s_{1}+1,s_{2}+1)}{s_{1}+s_{2}+1}M(a,b,c,d)  \notag \\
&&+\left[ B(s_{1}+1,s_{2}+1)\right] ^{2}N(a,b,c,d)  \notag \\
&=&\frac{1}{\left( s_{1}+s_{2}+1\right) ^{2}}\left[ L(a,b,c,d)+\frac{%
s_{1}s_{2}\Gamma (s_{1})\Gamma (s_{2})}{\Gamma (s_{1}+s_{2}+1)}%
M(a,b,c,d)\right.  \notag \\
&&\left. +\left[ \frac{s_{1}s_{2}\Gamma (s_{1})\Gamma (s_{2})}{\Gamma
(s_{1}+s_{2}+1)}\right] ^{2}N(a,b,c,d)\right]  \notag
\end{eqnarray}%
where%
\begin{eqnarray*}
L(a,b,c,d) &=&\left[ f(a,c)g(a,c)+f(b,c)g(b,c)+f(a,d)g(a,d)+f(b,d)g(b,d)%
\right] \\
M(a,b,c,d) &=&\left[ f(a,c)g(b,c)+f(b,c)g(a,c)+f(a,d)g(b,d)+f(b,d)g(a,d)%
\right] \\
&&+\left[ f(a,c)g(a,d)+f(b,c)g(b,d)+f(a,d)g(a,c)+f(b,d)g(b,c)\right] \\
N(a,b,c,d) &=&\left[ f(a,c)g(b,d)+f(b,c)g(a,d)+f(a,d)g(b,c)+f(b,d)g(a,c)%
\right]
\end{eqnarray*}
\end{theorem}

\begin{proof}
Since $f$ is co-ordinated $s_{1}-$convex and $g$ is co-ordinated $s_{2}-$%
convex, from Lemma 2, the partial mappings%
\begin{eqnarray*}
f_{y} &:&[a,b]\rightarrow \lbrack 0,\infty ),\text{ }f_{y}(x)=f(x,y) \\
f_{x} &:&[c,d]\rightarrow \lbrack 0,\infty ),\text{ }f_{x}(y)=f(x,y)
\end{eqnarray*}%
are $s_{1}-$convex on $[a,b]$ and $[c,d],$ respectively, where $x\in \lbrack
a,b],$ $y\in \lbrack c,d].$ Similarly;%
\begin{eqnarray*}
g_{y} &:&[a,b]\rightarrow \lbrack 0,\infty ),\text{ }g_{y}(x)=g(x,y) \\
g_{x} &:&[c,d]\rightarrow \lbrack 0,\infty ),g_{x}(y)=g(x,y)
\end{eqnarray*}%
are $s_{2}-$convex on $[a,b]$ and $[c,d],$ respectively, where $x\in \lbrack
a,b],$ $y\in \lbrack c,d].$

Using (\ref{1.3}), we get%
\begin{eqnarray*}
\frac{1}{d-c}\dint\limits_{c}^{d}f_{x}(y)g_{x}(y)dy &\leq &\frac{1}{%
s_{1}+s_{2}+1}\left[ f_{x}(c)g_{x}(c)+f_{x}(d)g_{x}(d)\right] \\
&&+B(s_{1}+1,s_{2}+1)\left[ f_{x}(c)g_{x}(d)+f_{x}(d)g_{x}(c)\right]
\end{eqnarray*}%
Therefore%
\begin{eqnarray*}
\frac{1}{d-c}\dint\limits_{c}^{d}f(x,y)g(x,y)dy &\leq &\frac{1}{s_{1}+s_{2}+1%
}\left[ f(x,c)g(x,c)+f(x,d)g(x,d)\right] \\
&&+B(s_{1}+1,s_{2}+1)\left[ f(x,c)g(x,d)+f(x,d)g(x,c)\right]
\end{eqnarray*}%
Dividing both sides of the above inequality by $(b-a)$ and integrating over $%
[a,b],$ we have%
\begin{eqnarray}
&&\frac{1}{\left( d-c\right) \left( b-a\right) }\dint\limits_{a}^{b}\dint%
\limits_{c}^{d}f(x,y)g(x,y)dxdy  \label{2.4} \\
&\leq &\frac{1}{s_{1}+s_{2}+1}\left[ \frac{1}{b-a}\dint%
\limits_{a}^{b}f(x,c)g(x,c)dx+\frac{1}{b-a}\dint\limits_{a}^{b}f(x,d)g(x,d)dx%
\right]  \notag \\
&&+B(s_{1}+1,s_{2}+1)\left[ \frac{1}{b-a}\dint\limits_{a}^{b}f(x,c)g(x,d)dx+%
\frac{1}{b-a}\dint\limits_{a}^{b}f(x,d)g(x,c)dx\right]  \notag
\end{eqnarray}%
By applying (\ref{1.3}) to right side of (\ref{2.4}), and we proceed
similarly as in the proof of Theorem 7, we can write%
\begin{eqnarray*}
&&\frac{1}{\left( d-c\right) \left( b-a\right) }\dint\limits_{a}^{b}\dint%
\limits_{c}^{d}f(x,y)g(x,y)dxdy \\
&\leq &\frac{1}{\left( s_{1}+s_{2}+1\right) ^{2}}\left[
f(a,c)g(a,c)+f(b,c)g(b,c)+f(a,d)g(a,d)+f(b,d)g(b,d)\right] \\
&&+\frac{B(s_{1}+1,s_{2}+1)}{s_{1}+s_{2}+1}\left[
f(a,c)g(b,c)+f(b,c)g(a,c)+f(a,d)g(b,d)+f(b,d)g(a,d)\right] \\
&&+\frac{B(s_{1}+1,s_{2}+1)}{s_{1}+s_{2}+1}\left[
f(a,c)g(a,d)+f(b,c)g(b,d)+f(a,d)g(a,c)+f(b,d)g(b,c)\right] \\
&&+\left[ B(s_{1}+1,s_{2}+1)\right] ^{2}\left[
f(a,c)g(b,d)+f(b,c)g(a,d)+f(a,d)g(b,c)+f(b,d)g(a,c)\right]
\end{eqnarray*}%
That is;%
\begin{eqnarray*}
&&\frac{1}{\left( d-c\right) \left( b-a\right) }\dint\limits_{a}^{b}\dint%
\limits_{c}^{d}f(x,y)g(x,y)dxdy \\
&\leq &\frac{1}{\left( s_{1}+s_{2}+1\right) ^{2}}L(a,b,c,d)+\frac{%
B(s_{1}+1,s_{2}+1)}{s_{1}+s_{2}+1}M(a,b,c,d) \\
&&+\left[ B(s_{1}+1,s_{2}+1)\right] ^{2}N(a,b,c,d) \\
&=&\frac{1}{\left( s_{1}+s_{2}+1\right) ^{2}}\left[ L(a,b,c,d)+\frac{%
s_{1}s_{2}\Gamma (s_{1})\Gamma (s_{2})}{\Gamma (s_{1}+s_{2}+1)}%
M(a,b,c,d)\right. \\
&&\left. +\left[ \frac{s_{1}s_{2}\Gamma (s_{1})\Gamma (s_{2})}{\Gamma
(s_{1}+s_{2}+1)}\right] ^{2}N(a,b,c,d)\right]
\end{eqnarray*}%
which completes the proof.
\end{proof}

\begin{remark}
In (\ref{2.3}) if we choose $s_{1}=s_{2}=1,$ (\ref{2.3}) reduces to (\ref%
{1.6}).
\end{remark}

\begin{theorem}
Let $\ f:\Delta :=$ $[a,b]\times \lbrack c,d]\subset \lbrack 0,\infty
)^{2}\rightarrow \lbrack 0,\infty )$ be convex function on the co-ordinates
and $g:\Delta :=$ $[a,b]\times \lbrack c,d]\subset \lbrack 0,\infty
)^{2}\rightarrow \lbrack 0,\infty )$ be $s-$convex function on the
co-ordinates with $a<b,$ $c<d$ and $f_{x}(y)g_{x}(y),$ $f_{y}(x)g_{y}(x)\in
L^{1}[\Delta ]$ for some fixed $s\in (0,1].$ Then one has the inequality:%
\begin{eqnarray}
&&2^{2s+1}f(\frac{a+b}{2},\frac{c+d}{2})g(\frac{a+b}{2},\frac{c+d}{2})
\label{2.5} \\
&\leq &\frac{2}{\left( b-a\right) (d-c)}\dint\limits_{a}^{b}\dint%
\limits_{c}^{d}f(x,y)g(x,y)dxdy  \notag \\
&&+\frac{5}{(s+1)(s+2)^{2}}L(a,b,c,d)+\frac{2s^{2}+6s+6}{(s+1)^{2}(s+2)^{2}}%
M(a,b,c,d)  \notag \\
&&+\frac{2s+6}{(s+1)(s+2)^{2}}N(a,b,c,d)  \notag
\end{eqnarray}
\end{theorem}

\begin{proof}
Since $f$ is co-ordinated convex and $g$ is co-ordinated $s-$convex, from
Lemma 1 and Lemma 2, the partial mappings%
\begin{eqnarray}
f_{y} &:&[a,b]\rightarrow \lbrack 0,\infty ),\text{ }f_{y}(x)=f(x,y)
\label{2.6} \\
f_{x} &:&[c,d]\rightarrow \lbrack 0,\infty ),\text{ }f_{x}(y)=f(x,y)  \notag
\end{eqnarray}%
are convex on $[a,b]$ and $[c,d],$ respectively, where $x\in \lbrack a,b],$ $%
y\in \lbrack c,d]$. Similarly;%
\begin{eqnarray*}
g_{y} &:&[a,b]\rightarrow \lbrack 0,\infty ),\text{ }g_{y}(x)=g(x,y),\text{ }%
y\in \lbrack c,d] \\
g_{x} &:&[c,d]\rightarrow \lbrack 0,\infty ),\text{ }g_{x}(y)=g(x,y),\text{ }%
x\in \lbrack a,b]
\end{eqnarray*}%
are $s-$convex on $[a,b]$ and $[c,d],$ respectively, where $x\in \lbrack
a,b],$ $y\in \lbrack c,d].$

Using (\ref{1.4}) and multiplying both sides of the inequalities by $2^{s},$
we get 
\begin{eqnarray}
&&2^{2s}f(\frac{a+b}{2},\frac{c+d}{2})g(\frac{a+b}{2},\frac{c+d}{2})
\label{2.7} \\
&&-\frac{2^{s}}{b-a}\dint\limits_{a}^{b}f(x,\frac{c+d}{2})g(x,\frac{c+d}{2}%
)dx  \notag \\
&\leq &\frac{2^{s}}{(s+1)(s+2)}\left[ f(a,\frac{c+d}{2})g(a,\frac{c+d}{2}%
)+f(b,\frac{c+d}{2})g(b,\frac{c+d}{2})\right]  \notag \\
&&+\frac{2^{s}}{s+2}\left[ f(a,\frac{c+d}{2})g(b,\frac{c+d}{2})+f(b,\frac{c+d%
}{2})g(a,\frac{c+d}{2})\right]  \notag
\end{eqnarray}%
and%
\begin{eqnarray}
&&2^{2s}f(\frac{a+b}{2},\frac{c+d}{2})g(\frac{a+b}{2},\frac{c+d}{2})
\label{2.8} \\
&&-\frac{2^{s}}{d-c}\dint\limits_{c}^{d}f(\frac{a+b}{2},y)g(\frac{a+b}{2}%
,y)dy  \notag \\
&\leq &\frac{2^{s}}{(s+1)(s+2)}\left[ f(\frac{a+b}{2},c)g(\frac{a+b}{2},c)+f(%
\frac{a+b}{2},d)g(\frac{a+b}{2},d)\right]  \notag \\
&&+\frac{2^{s}}{s+2}\left[ f(\frac{a+b}{2},c)g(\frac{a+b}{2},d)+f(\frac{a+b}{%
2},d)g(\frac{a+b}{2},c)\right]  \notag
\end{eqnarray}%
Now, on adding (\ref{2.7}) and (\ref{2.8}), we get 
\begin{eqnarray}
&&2^{2s+1}f(\frac{a+b}{2},\frac{c+d}{2})g(\frac{a+b}{2},\frac{c+d}{2})
\label{2.9} \\
&&-\frac{2^{s}}{b-a}\dint\limits_{a}^{b}f(x,\frac{c+d}{2})g(x,\frac{c+d}{2}%
)dx-\frac{2^{s}}{d-c}\dint\limits_{c}^{d}f(\frac{a+b}{2},y)g(\frac{a+b}{2}%
,y)dy  \notag \\
&\leq &\frac{1}{(s+1)(s+2)}\left[ 2^{s}f(a,\frac{c+d}{2})g(a,\frac{c+d}{2}%
)+2^{s}f(b,\frac{c+d}{2})g(b,\frac{c+d}{2})\right]  \notag \\
&&+\frac{1}{s+2}\left[ 2^{s}f(a,\frac{c+d}{2})g(b,\frac{c+d}{2})+2^{s}f(b,%
\frac{c+d}{2})g(a,\frac{c+d}{2})\right]  \notag \\
&&+\frac{1}{(s+1)(s+2)}\left[ 2^{s}f(\frac{a+b}{2},c)g(\frac{a+b}{2}%
,c)+2^{s}f(\frac{a+b}{2},d)g(\frac{a+b}{2},d)\right]  \notag \\
&&+\frac{1}{s+2}\left[ 2^{s}f(\frac{a+b}{2},c)g(\frac{a+b}{2},d)+2^{s}f(%
\frac{a+b}{2},d)g(\frac{a+b}{2},c)\right]  \notag
\end{eqnarray}%
Applying (\ref{1.4}) to each term of right hand side of the above
inequality, we have%
\begin{eqnarray*}
&&2^{s}f(a,\frac{c+d}{2})g(a,\frac{c+d}{2}) \\
&\leq &\frac{1}{d-c}\dint\limits_{c}^{d}f(a,y)g(a,y)dy+\frac{1}{(s+1)(s+2)}%
\left[ f(a,c)g(a,c)+f(a,d)g(a,d)\right] \\
&&+\frac{1}{s+2}\left[ f(a,c)g(a,d)+f(a,d)g(a,c)\right]
\end{eqnarray*}%
\begin{eqnarray*}
&&2^{s}f(b,\frac{c+d}{2})g(b,\frac{c+d}{2}) \\
&\leq &\frac{1}{d-c}\dint\limits_{c}^{d}f(b,y)g(b,y)dy+\frac{1}{(s+1)(s+2)}%
\left[ f(b,c)g(b,c)+f(b,d)g(b,d)\right] \\
&&+\frac{1}{s+2}\left[ f(b,c)g(b,d)+f(b,d)g(b,c)\right]
\end{eqnarray*}%
\begin{eqnarray*}
&&2^{s}f(a,\frac{c+d}{2})g(b,\frac{c+d}{2}) \\
&\leq &\frac{1}{d-c}\dint\limits_{c}^{d}f(a,y)g(b,y)dy+\frac{1}{(s+1)(s+2)}%
\left[ f(a,c)g(b,c)+f(a,d)g(b,d)\right] \\
&&+\frac{1}{s+2}\left[ f(a,c)g(b,d)+f(a,d)g(b,c)\right]
\end{eqnarray*}%
\begin{eqnarray*}
&&2^{s}f(b,\frac{c+d}{2})g(a,\frac{c+d}{2}) \\
&\leq &\frac{1}{d-c}\dint\limits_{c}^{d}f(b,y)g(a,y)dy+\frac{1}{(s+1)(s+2)}%
\left[ f(b,c)g(a,c)+f(b,d)g(a,d)\right] \\
&&+\frac{1}{s+2}\left[ f(b,c)g(a,d)+f(b,d)g(a,c)\right]
\end{eqnarray*}%
\begin{eqnarray*}
&&2^{s}f(\frac{a+b}{2},c)g(\frac{a+b}{2},c) \\
&\leq &\frac{1}{b-a}\dint\limits_{a}^{b}f(x,c)g(x,c)dx+\frac{1}{(s+1)(s+2)}%
\left[ f(a,c)g(a,c)+f(b,c)g(b,c)\right] \\
&&+\frac{1}{s+2}\left[ f(a,c)g(b,c)+f(b,c)g(a,c)\right]
\end{eqnarray*}%
\begin{eqnarray*}
&&2^{s}f(\frac{a+b}{2},d)g(\frac{a+b}{2},d) \\
&\leq &\frac{1}{b-a}\dint\limits_{a}^{b}f(x,d)g(x,d)dx+\frac{1}{(s+1)(s+2)}%
\left[ f(a,d)g(a,d)+f(b,d)g(b,d)\right] \\
&&+\frac{1}{s+2}\left[ f(a,d)g(b,d)+f(b,d)g(a,d)\right]
\end{eqnarray*}%
\begin{eqnarray*}
&&2^{s}f(\frac{a+b}{2},c)g(\frac{a+b}{2},d) \\
&\leq &\frac{1}{b-a}\dint\limits_{a}^{b}f(x,c)g(x,d)dx+\frac{1}{(s+1)(s+2)}%
\left[ f(a,c)g(a,d)+f(b,c)g(b,d)\right] \\
&&+\frac{1}{s+2}\left[ f(a,c)g(b,d)+f(b,c)g(a,d)\right]
\end{eqnarray*}%
\begin{eqnarray*}
&&2^{s}f(\frac{a+b}{2},d)g(\frac{a+b}{2},c) \\
&\leq &\frac{1}{b-a}\dint\limits_{a}^{b}f(x,d)g(x,c)dx+\frac{1}{(s+1)(s+2)}%
\left[ f(a,d)g(a,c)+f(b,d)g(b,c)\right] \\
&&+\frac{1}{s+2}\left[ f(a,d)g(b,c)+f(b,d)g(a,c)\right]
\end{eqnarray*}%
Using these inequalities in (\ref{2.9}), we have%
\begin{eqnarray}
&&2^{2s+1}f(\frac{a+b}{2},\frac{c+d}{2})g(\frac{a+b}{2},\frac{c+d}{2})
\label{2.10} \\
&&-\frac{2^{s}}{b-a}\dint\limits_{a}^{b}f(x,\frac{c+d}{2})g(x,\frac{c+d}{2}%
)dx-\frac{2^{s}}{d-c}\dint\limits_{c}^{d}f(\frac{a+b}{2},y)g(\frac{a+b}{2}%
,y)dy  \notag \\
&\leq &\frac{1}{(s+1)(s+2)}\frac{1}{\left( d-c\right) }\left[
\dint\limits_{c}^{d}f(a,y)g(a,y)dy+\dint\limits_{c}^{d}f(b,y)g(b,y)dy\right]
\notag \\
&&+\frac{1}{\left( s+2\right) }\frac{1}{\left( d-c\right) }\left[
\dint\limits_{c}^{d}f(a,y)g(b,y)dy+\dint\limits_{c}^{d}f(b,y)g(a,y)dy\right]
\notag \\
&&+\frac{1}{(s+1)(s+2)}\frac{1}{\left( b-a\right) }\left[ \dint%
\limits_{a}^{b}f(x,c)g(x,c)dx+\dint\limits_{a}^{b}f(x,d)g(x,d)dx\right] 
\notag \\
&&+\frac{1}{\left( s+2\right) }\frac{1}{\left( b-a\right) }\left[
\dint\limits_{a}^{b}f(x,c)g(x,d)dx+\dint\limits_{a}^{b}f(x,d)g(x,c)dx\right]
\notag \\
&&+\frac{2}{(s+1)^{2}(s+2)^{2}}L(a,b,c,d)+\frac{2}{(s+1)\left( s+2\right)
^{2}}M(a,b,c,d)  \notag \\
&&+\frac{2}{\left( s+2\right) ^{2}}N(a,b,c,d)  \notag
\end{eqnarray}%
Now by applying (\ref{1.4}) to $2^{s}f(\frac{a+b}{2},y)g(\frac{a+b}{2},y),$
integrating over $[c,d]$, and dividing both sides by $(d-c),$ we get 
\begin{eqnarray}
&&\frac{2^{s}}{(d-c)}\dint\limits_{c}^{d}f(\frac{a+b}{2},y)g(\frac{a+b}{2}%
,y)dy  \label{2.11} \\
&&-\frac{1}{\left( b-a\right) (d-c)}\dint\limits_{a}^{b}\dint%
\limits_{c}^{d}f(x,y)g(x,y)dxdy  \notag \\
&\leq &\frac{1}{(s+1)(s+2)}\left[ \frac{1}{(d-c)}\dint%
\limits_{c}^{d}f(a,y)g(a,y)dy+\frac{1}{(d-c)}\dint%
\limits_{c}^{d}f(b,y)g(b,y)dy\right]  \notag \\
&&+\frac{1}{s+2}\left[ \frac{1}{(d-c)}\dint\limits_{c}^{d}f(a,y)g(b,y)dy+%
\frac{1}{(d-c)}\dint\limits_{c}^{d}f(b,y)g(a,y)dy\right]  \notag
\end{eqnarray}%
Similarly by applying (\ref{1.4}) to $2^{s}f(x,\frac{c+d}{2})g(x,\frac{c+d}{2%
}),$ integrating over $[a,b]$, dividing both sides by $(b-a),$ we get%
\begin{eqnarray}
&&\frac{2^{s}}{(b-a)}\dint\limits_{a}^{b}f(x,\frac{c+d}{2})g(x,\frac{c+d}{2}%
)dx  \label{2.12} \\
&&-\frac{1}{\left( b-a\right) (d-c)}\dint\limits_{a}^{b}\dint%
\limits_{c}^{d}f(x,y)g(x,y)dxdy  \notag \\
&\leq &\frac{1}{(s+1)(s+2)}\left[ \frac{1}{\left( b-a\right) }%
\dint\limits_{a}^{b}f(x,c)g(x,c)dx+\frac{1}{\left( b-a\right) }%
\dint\limits_{a}^{b}f(x,d)g(x,d)dx\right]  \notag \\
&&+\frac{1}{s+2}\left[ \frac{1}{\left( b-a\right) }\dint%
\limits_{a}^{b}f(x,c)g(x,d)dx+\frac{1}{\left( b-a\right) }%
\dint\limits_{a}^{b}f(x,d)g(x,c)dx\right]  \notag
\end{eqnarray}%
By addition (\ref{2.11}) and (\ref{2.12}), we have%
\begin{eqnarray}
&&\frac{2^{s}}{(d-c)}\dint\limits_{c}^{d}f(\frac{a+b}{2},y)g(\frac{a+b}{2}%
,y)dy+\frac{2^{s}}{(b-a)}\dint\limits_{a}^{b}f(x,\frac{c+d}{2})g(x,\frac{c+d%
}{2})dx  \notag \\
&&-\frac{2}{\left( b-a\right) (d-c)}\dint\limits_{a}^{b}\dint%
\limits_{c}^{d}f(x,y)g(x,y)dxdy  \label{2.13} \\
&\leq &\frac{1}{(s+1)(s+2)}\left[ \frac{1}{(d-c)}\dint%
\limits_{c}^{d}f(a,y)g(a,y)dy+\frac{1}{(d-c)}\dint%
\limits_{c}^{d}f(b,y)g(b,y)dy\right.  \notag \\
&&\left. +\frac{1}{\left( b-a\right) }\dint\limits_{a}^{b}f(x,c)g(x,c)dx+%
\frac{1}{\left( b-a\right) }\dint\limits_{a}^{b}f(x,d)g(x,d)dx\right]  \notag
\\
&&+\frac{1}{s+2}\left[ \frac{1}{(d-c)}\dint\limits_{c}^{d}f(a,y)g(b,y)dy+%
\frac{1}{(d-c)}\dint\limits_{c}^{d}f(b,y)g(a,y)dy\right.  \notag \\
&&\left. +\frac{1}{\left( b-a\right) }\dint\limits_{a}^{b}f(x,c)g(x,d)dx+%
\frac{1}{\left( b-a\right) }\dint\limits_{a}^{b}f(x,d)g(x,c)dx\right]  \notag
\end{eqnarray}%
From (\ref{2.10}) and (\ref{2.13}) and simplifying we get%
\begin{eqnarray*}
&&2^{2s+1}f(\frac{a+b}{2},\frac{c+d}{2})g(\frac{a+b}{2},\frac{c+d}{2})\leq 
\frac{2}{\left( b-a\right) (d-c)}\dint\limits_{a}^{b}\dint%
\limits_{c}^{d}f(x,y)g(x,y)dx \\
&&+\frac{4s+6}{(s+1)^{2}(s+2)^{2}}L(a,b,c,d)+\frac{2s^{2}+6s+6}{%
(s+1)^{2}(s+2)^{2}}M(a,b,c,d) \\
&&+\frac{2s^{2}+8s+6}{(s+1)^{2}(s+2)^{2}}N(a,b,c,d)
\end{eqnarray*}
\end{proof}

\begin{remark}
In (\ref{2.5}), if we choose $s=1,$ we obtained (\ref{1.7}).
\end{remark}

\begin{remark}
In (\ref{2.5}), if we choose $s=1$ and $f(x)=1$ which is convex, we have the
following Hadamard-type inequality like (\ref{1.5}) 
\begin{eqnarray*}
&&4g(\frac{a+b}{2},\frac{c+d}{2})-\frac{1}{\left( b-a\right) (d-c)}%
\dint\limits_{a}^{b}\dint\limits_{c}^{d}g(x,y)dx \\
&\leq &\frac{3\left[ g(a,c)+g(b,c)+g(a,d)+g(b,d)\right] }{4}
\end{eqnarray*}
\end{remark}

\begin{theorem}
Let $\ f,g:\Delta :=$ $[a,b]\times \lbrack c,d]\subset 
%TCIMACRO{\U{211d} }%
%BeginExpansion
\mathbb{R}
%EndExpansion
^{2}\rightarrow 
%TCIMACRO{\U{211d} }%
%BeginExpansion
\mathbb{R}
%EndExpansion
$ be convex function on the co-ordinates with $a<b,$ $c<d$ and $%
f_{x}(y)g_{x}(y),$ $\ f_{y}(x)g_{y}(x)\in L^{1}[\Delta ].$ Then one has the
inequality:%
\begin{eqnarray*}
&&\frac{1}{\left( b-a\right) ^{2}\left( d-c\right) ^{2}}\left[ f\left(
a,c\right) \dint\limits_{a}^{b}\dint\limits_{c}^{d}\left( x-b\right) \left(
y-d\right) g(x,y)dydx\right. \\
&&+f\left( b,c\right) \dint\limits_{a}^{b}\dint\limits_{c}^{d}\left(
a-x\right) \left( y-d\right) g(x,y)dydx+f\left( a,d\right)
\dint\limits_{a}^{b}\dint\limits_{c}^{d}\left( x-b\right) \left( c-y\right)
g(x,y)dydx \\
&&+f\left( b,d\right) \dint\limits_{a}^{b}\dint\limits_{c}^{d}\left(
a-x\right) \left( c-y\right) g(x,y)dydx+g\left( a,c\right)
\dint\limits_{a}^{b}\dint\limits_{c}^{d}\left( x-b\right) \left( y-d\right)
f(x,y)dydx \\
&&+g\left( b,c\right) \dint\limits_{a}^{b}\dint\limits_{c}^{d}\left(
a-x\right) \left( y-d\right) f(x,y)dydx+g\left( a,d\right)
\dint\limits_{a}^{b}\dint\limits_{c}^{d}\left( x-b\right) \left( c-y\right)
f(x,y)dydx \\
&&\left. +g\left( b,d\right) \dint\limits_{a}^{b}\dint\limits_{c}^{d}\left(
a-x\right) \left( c-y\right) f(x,y)dydx\right] \\
&\leq &\frac{1}{\left( b-a\right) (d-c)}\dint\limits_{a}^{b}\dint%
\limits_{c}^{d}f(x,y)g(x,y)dx+\frac{1}{9}L(a,b,c,d)+\frac{1}{18}M(a,b,c,d)+%
\frac{1}{36}N(a,b,c,d)
\end{eqnarray*}%
where $L(a,b,c,d),$ $M(a,b,c,d),$ $N(a,b,c,d)$ defined as in Theorem 6.
\end{theorem}

\begin{proof}
Since $f$ and $g$ are co-ordinated convex functions on the co-ordinates on $%
\Delta $, from the definition of co-ordinated convexity, we can write%
\begin{equation*}
f(ta+(1-t)b,sc+(1-s)d)\leq
tsf(a,c)+t(1-s)f(a,d)+s(1-t)f(b,c)+(1-t)(1-s)f(b,d)
\end{equation*}%
and%
\begin{equation*}
g(ta+(1-t)b,sc+(1-s)d)\leq
tsg(a,c)+t(1-s)g(a,d)+s(1-t)g(b,c)+(1-t)(1-s)g(b,d)
\end{equation*}
holds for all $t,s\in \lbrack 0,1]$. By using the elementary inequality, if $%
e\leq f$ and $p\leq r,$ then $er+fp\leq ep+fr$ for all $e,f,p,r\in 
%TCIMACRO{\U{211d} }%
%BeginExpansion
\mathbb{R}
%EndExpansion
,$ we get%
\begin{eqnarray*}
&&f(ta+(1-t)b,sc+(1-s)d) \\
&&\times \left[ tsg(a,c)+t(1-s)g(a,d)+s(1-t)g(b,c)+(1-t)(1-s)g(b,d)\right] \\
&&+g(ta+(1-t)b,sc+(1-s)d) \\
&&\times \left[ tsf(a,c)+t(1-s)f(a,d)+s(1-t)f(b,c)+(1-t)(1-s)f(b,d)\right] \\
&\leq &\left[ f(ta+(1-t)b,sc+(1-s)d)g(ta+(1-t)b,sc+(1-s)d)\right] \\
&&+\left[ tsf(a,c)+t(1-s)f(a,d)+s(1-t)f(b,c)+(1-t)(1-s)f(b,d)\right] \\
&&\times \left[ tsg(a,c)+t(1-s)g(a,d)+s(1-t)g(b,c)+(1-t)(1-s)g(b,d)\right] .
\end{eqnarray*}%
By integrating the above integral on $\left[ 0,1\right] \times \left[ 0,1%
\right] $, with respect to $t$, $s$ and by taking into account the change of
variables $ta+(1-t)b=x,$ $(a-b)dt=dx$ and $sc+(1-s)d=y,$ $(c-d)ds=dy$, we
obtain the desired result.
\end{proof}

\end{document}